\documentclass{llncs}

\usepackage{graphicx}
\usepackage{amsmath}
\usepackage{amssymb}
\usepackage{fancyhdr}

\newcommand\rmgamma{{\mathrm \Gamma}}

\title{Fully Analyzing an Algebraic P\'olya Urn Model}
\author{Basile Morcrette}
\institute
{
ALGORITHMS Project, INRIA Paris-Rocquencourt, 78153 Le Chesnay (France) \\ LIP6 , Universit\'e Paris 6, 4 place Jussieu, 75005 Paris (France) \\
\email{Basile.Morcrette@inria.fr}}

\begin{document}
\lhead{Fully Analyzing an Algebraic P\'olya Urn Model}
\rhead{B. Morcrette}

\pagestyle{fancy}

\maketitle

\begin{abstract}
This paper introduces and analyzes a particular class of P\'olya urns: balls are of two colors, can only be added (the urns are said to be \emph{additive}) and at every step the same constant number of balls is added, thus only the color compositions varies (the urns are said to be \emph{balanced}). These properties make this class of urns ideally suited for analysis from an ``analytic combinatorics'' point-of-view, following in the footsteps of Flajolet et al. \cite{FlaDuPuy06}. Through an \emph{algebraic} generating function to which we apply a multiple coalescing saddle-point method, we are able to give precise asymptotic results for the probability distribution of the composition of the urn, as well as local limit law and large deviation bounds.
\end{abstract}
\begin{flushright}
{\it Dedicated to the memory of Philippe Flajolet.}
\end{flushright}


\section{Introduction}

A P\'olya urn is an urn which contains balls of two colors (black and white), and which is coupled with an initial configuration and a set of evolution rules. A step then consists in randomly picking a ball from the urn, placing it back, and depending on its color, adding a fixed number of black and/or white balls.
The question is: what does the urn look like after a large number of steps? This simple process has turned out to be extremely versatile, and has been used to model many different phenomena, such as population growth, epidemics, tree structures in computer science (BST, $(a,b)$-trees), electoral campaigns, etc.

This paper analyzes a class of balanced additive urns. \emph{Balanced} urns are urns for which, at every step, the same constant number of balls is added. This property allows us to resort to a combinatorial treatment: enumerating all configurations using generating functions. Such an approach, introduced by Flajolet and his coauthors \cite{FlaDuPuy06}, is a departure from previous probabilistic methods. In \emph{additive} urns, no ball is ever removed and the urn's size is strictly increasing. We specifically consider a class of balanced additive urns which has \emph{algebraic} generating functions.

Through the use of analytic combinatorics \cite{FlaSed09}, we obtain precise probability results, including a Gaussian limit law with rate of convergence, a local limit law and large deviation bounds. Our analysis makes use of a multiple coalescing saddle-point method, which is not classical. As previously thought, analytic combinatorial methods can provide a wealth of valuable information that seems to be new. In addition, our results seem to reinforce the idea of a general theory for additive balanced urns.

This paper comes as the continuation of Flajolet's work on urns \cite{FlaGaPe05} and \cite{FlaDuPuy06}; these two papers were on subtractive and triangular urn models. Of course, this topic has been thoroughly addressed through probabilistic means. We can cite introductory books on urns \cite{JoKo77} \cite{Mah08} and an article \cite{BagchiPal}; on the topic of limit distribution for urns, Janson's papers are a reference \cite{Jan04} \cite{Jan05} \cite{Jan08} as is Smythe's \cite{Smythe96}; and on the topic of large additive urns, \cite{ChauPouSah09}. From an analytic point of view, Hwang et al.'s paper \cite{HwKuPa07} considers ``\emph{diminishing}'' urn models. And in direct relation to the present paper, we have obtained similar results \cite{Morcrette10} on a class of additive balanced urns which is linked to some family of $k$-trees \cite{PanSei10}.

In Sect.~{\ref{urn_system}}, we lay the ground work by introducing P\'olya urn models as well as the main result of \cite{FlaDuPuy06} on the isomorphism between balanced urns and differential systems. Then, in Sect.~{\ref{note1}}, we introduce our urn class, which can be viewed as a population growth model. We establish that the generating function is algebraic and provide our first results on mean and variance. Section~{\ref{limit_th}} contains the main theorems of the paper regarding the limit distribution of our urn class. Finally we state the propositions and lemmas required by the proof of the main theorems, and give some extensive details on the part involving multiple coalescing saddle-points.


\section{Urns and Differential Systems}\label{urn_system}

In P\'olya's classical urn model, we have an urn containing balls of two different colors, black balls ({\tt b} type) and white balls ({\tt w}~type). This system evolves with regards to particular rules (at each step, add and/or discard black and/or white balls), and these rules are specified by a $2\times2$ matrix
\begin{equation} \label{urne22}
\begin{pmatrix}
a & b \\
c & d  \end{pmatrix} \ \ \ a,  d \in \mathbb{Z},\ \ \ \  b, c \in \mathbb{Z}_{\geqslant0} \, .
\end{equation}
We start with an initial configuration $(a_0, b_0)$. At step $0$, the urn contains $a_0$ black balls and $b_0$ white balls. The evolution between steps $n$ and $n+1$ is now described. We uniformly draw a ball from the urn, we look at its color and \emph{we put it back into the urn}. If the color is black, then we add $a$ black balls and $b$ white balls; if the color is white, we add $c$ black balls and $d$ white balls.

\begin{definition}\label{def_balance}
The urn (\ref{urne22}) is said to be \emph{balanced} if the sums of its rows are constant, that is if $a+b = c + d$. This parameter is called the \emph{balance} of the urn, and denoted by $\sigma$. 
\end{definition}
\begin{remark}
At each step, we add $\sigma$ balls in the urn. So, starting with $a_0 + b_0$ balls in a balanced urn, we know that the total number of balls after $n$ steps will be $a_0+b_0 + \sigma n$. This balanced condition is the key requirement for using Flajolet-Dumas-Puyhaubert differential systems. 
\end{remark}

\begin{definition}\label{def_additive}
An urn is \emph{additive} if all the coefficients in its rule matrix (\ref{urne22}) are strictly positive, i.e., $a, b, c, d > 0$.
\end{definition}
\begin{remark} No balls are ever removed from an additive urn.
In our study, we focus on \emph{balanced additive} urns.\end{remark}

The papers \cite{FlaGaPe05} and \cite{FlaDuPuy06} introduce a new analytical and combinatorial approach to the study of these balanced urn models. The crucial starting point is an isomorphism theorem between the rules of an urn and a differential system. We briefly recall these results.
\begin{definition}\label{def_histoire}
A \emph{history} of length $n$ is a sequence of $n$ steps, obtained by $n$ successive draws from the urn. 
The \emph{exponential generating function of histories} of the urn (\ref{urne22}) is
\begin{equation}\label{genfun}
H(x,y,z;a_0,b_0) = \sum_{n,i,j} H_n (a_0,b_0;i,j) \, x^i y^j  \frac{z^n}{n !}\, ,
\end{equation}
where $H_n(a_0, b_0; i, j)$ is the number of histories of length $n$ starting at the configuration $(a_0,b_0)$, and ending at a configuration $(i,j)$.
We will shorten this to $H(x,y,z)$ when there is no ambiguity on the initial configuration $(a_0,b_0)$.
\end{definition}

\begin{theorem}[Flajolet--Dumas--Puyhaubert]\label{theo_fond}
We associate to the urn (\ref{urne22}), with the balanced condition $a+b = c+d$, the following differential system (we denote $t$-differentiation by a point, $\dot{X}(t) = \frac{d}{dt} X(t)$):
\begin{center}
$\Bigg\{$
\begin{tabular}{cccc}
$\dot{X}$  & $=$ & $X^{a+1}\, Y^{b}$& \\
$\dot{Y}$  & $=$ & $X^{c}\, Y^{d +1}$& .
\end{tabular}
\end{center}
Let $x_0$ and $y_0$ be two complex variables such that $x_0 y_0 \neq 0$. Let $X(x_0,y_0,t)$ and $Y(x_0,y_0,t)$ be the solutions of the differential system with initial conditions $X(x_0, y_0,0)=x_0,\ Y(x_0,y_0,0)=y_0$. Then the generating function of histories is given by
\begin{equation}\label{}
H(x_0,y_0,z) = X(x_0,y_0,z)^{a_0} \, Y(x_0,y_0,z)^{b_0} \, .
\end{equation}
\end{theorem}


\section{Preferential Growth Urns}\label{note1}

As previously mentioned, \cite{FlaGaPe05} has signaled that analytic methods can be used in the treatment of urn models. Since then, a general theory has been described for several urn models: for two-color balanced subtractive models and triangular models \cite{FlaDuPuy06}, as well as for some particular unbalanced models \cite{HwKuPa07}. The extension of these methods for all additive urn models is an open problem. This paper presents a full asymptotic study for some restriction of the balanced additive models. More specifically, here we present a two-parameter urn model, whereas the entire class of balanced additive models is described with three parameters. Indeed, knowing the four matrix coefficients and the balance hypothesis, one parameter is redundant. Our class is studied through its histories generating function, which is in fact \emph{algebraic} in our case.
\begin{definition}\label{our_class}
The class of \emph{preferential growth urns} denoted by $\mathcal{A}(\alpha, \beta)$ is defined by the two-color balanced matrix
\begin{equation}\label{urne_note1}
\begin{pmatrix}2 \alpha & \beta \\ \alpha & \alpha + \beta \end{pmatrix}, \mbox{ with } \alpha > 0, \,  \beta > 0   \ .
\end{equation}
\end{definition}

\begin{example} 
A first example $\mathcal{A}(1, 1)$ corresponds to the following evolution rules:
\begin{center}
$\begin{array}{ccc}
\texttt{b} & \rightarrow & \texttt{b} \, \texttt{b} \, \texttt{b} \, \texttt{w} \\
\texttt{w} & \rightarrow & \texttt{b} \, \texttt{w} \, \texttt{w} \, \texttt{w}
\end{array}$ \quad
corresponding to the urn \quad $\begin{pmatrix} 2 & 1 \\ 1 & 2 \end{pmatrix}$.
\end{center}
\noindent If we pick a {\tt{b}} ball, we replace it by three {\tt{b}} (counting the ball we drew and are now replacing) and one {\tt{w}}. If we pick a {\tt{w}} ball, we replace it by one {\tt{b}} and three {\tt{w}}.
In this particular example, we can see a model of population growth with two types of individuals ({\tt{b}} and {\tt{w}}). Every individual has three children and two of them are of the same type as their parent. Every individual encourages its own type with the ratio 2/1.
From this example, we choose to name our class \emph{preferential growth urns}.
\end{example}
\begin{remark}  Here are some characteristics of this urn class. The balance of $\mathcal{A}(\alpha, \beta)$ is $\sigma:= 2 \alpha + \beta$. The dissymetry index is defined by $p:=\alpha - 2\alpha = \beta - (\alpha + \beta) = -\alpha$.\\
The two eigenvalues of the matrix are $-p$ and $\sigma$. Besides, the ratio $\rho$ of these two eigenvalues is $\rho  :=  \frac{-p}{\sigma} = \frac{\alpha}{2\alpha + \beta} \leqslant \frac{1}{2}.$ It is a small urn so, from the probabilistic analyses of Smythe \cite{Smythe96} and Janson \cite{Jan04}, we already have a Gaussian behavior for the limiting distribution of the two colored balls in the urn.
\end{remark}
In this paper, thanks to analytic combinatorics, first we obtain concrete and precise asymptotic results with rate of convergence, local limit laws and probabilities of large deviations; second, the analytical proof deals with a multiple coalescing saddle-point method; and third, this is the first step towards a full study of additive balanced urns.
%
\subsection{An Algebraic Generating Function}\label{note1_gf}
We start the urn process $\mathcal{A}(\alpha,\beta)$ with no black balls $\texttt{b}$, and one white ball $\texttt{w}$. So, in the following, $(a_0,b_0)=(0,1)$. The main result of this subsection is exhibiting the algebraic nature of the generating function of histories.
\begin{theorem}\label{gf_algebric}
The bivariate generating function of histories $H(x,z) := H(x,1,z)$ is \emph{algebraic}, and the following polynomial in $y$ cancels it.\footnote{In this paper, we only treat the case $(a_0,b_0)=(0,1)$. The general case $H=X^{a_0}Y^{b_0}$ derives directly from this study, because $X$ and $Y$ are bound: \mbox{$H(x,y,z)=\left(x^{-\alpha} - y^{-\alpha} + Y^{-\alpha}\right)^{-a_0/\alpha} Y^{b_0}$}. We focus here on the function $Y$.}
\begin{equation} \label{eqn_algebric}
(z-A-B_x) \, y^{2\alpha + \beta} + B_x \,  y^{\alpha} + A = 0 \, , 
\end{equation}
$$\mbox{where}\qquad A := \frac{1}{2\alpha + \beta}=\sigma^{-1} \qquad\text{and}\qquad B_x  :=  \frac{x^{-\alpha} - 1}{\alpha + \beta}  \, .$$
\end{theorem}

\begin{proof}
The exponential generating function of histories is $$H(x,y,z) = \sum_{n,i,j} H_n (i,j)\,  x^i y^j \frac{z^
n}{n !} \, ,$$
and the differential system associated to the urn is 
\begin{center}
$\Bigg\{$
\begin{tabular}{ccc}
$\dot{X}$  & $=$ & $X^{2\alpha +1}\, Y^{\beta}$ \\
$\dot{Y}$  & $=$ &  $X^{\alpha}\, Y^{\alpha + \beta +1} \, .$
\end{tabular}
\end{center}
Let $\left(X(x_0,y_0,t), Y(x_0,y_0,t)\right)$ be the associated solution, then thanks to Theorem~\ref{theo_fond}, and $a_0=0$ , $b_0=1$, we get
$H(x_0,y_0,) = Y(x_0,y_0,z)$.
Rewriting the system, $$\frac{\dot{X}}{X^{\alpha +1}} = X^{\alpha}Y^{\beta} = \frac{\dot{Y}}{Y^{\alpha +1}} \, , \qquad \text{so, } \qquad X^{-\alpha}  - Y^{-\alpha} = x_0^{-\alpha} - y_0^{-\alpha} \, . $$
Then, $$\frac{\dot{Y}}{Y^{\alpha + \beta + 1}} \left(Y^{-\alpha}  + x_0^{-\alpha} - y_0^{-\alpha} \right) = 1 \, . $$
After a $z$-integration and naming the integration constant $K(x_0,y_0)$, $$\frac{1}{2\alpha + \beta}Y^{-(2\alpha + \beta)} + \frac{x_0^{-\alpha} - y_0^{-\alpha}}{\alpha + \beta} Y^{-(\alpha + \beta)} = - (z - K(x_0,y_0)) \, . $$
We can set $y_0=1$ (the variable is useless because of the balanced condition. Indeed, in $H$, the non-negative coefficients of $x^iy^j z^n$ appear when $i+j = a_0+b_0 +\sigma n$), then we get the result.
Thus, the generating function $H(x,z):=H(x,1,z)$ is algebraic, as solution of a polynomial of degree $2\alpha + \beta$. \qed
\end{proof}
\subsection{Mean and Variance}\label{note1_moy}
From the polynomial equation (\ref{eqn_algebric}), we can directly extract some precise information concerning the expression of the mean and the variance of the urn composition. It is also possible to get all moments.
\begin{proposition}\label{moy_var}
Let $X_n$ be the random variable counting the number of black balls {\tt b} in the urn (\ref{urne_note1}) at step $n$. Then we can express mean and variance asymptotically: \footnote{\emph{Exact} expressions for mean and variance are also computable from this, see \cite{FlaSed09} (Prop.III.2, {\it p. 158} and {\it 728} ).}
\begin{equation}
{\mathbb{E}}(X_n) = \frac{\alpha (2\alpha + \beta)}{\alpha + \beta} n + \frac{\alpha}{\alpha + \beta} \frac{\rmgamma (\frac{1}{2\alpha + \beta})}{\rmgamma (\frac{\alpha + 1}{2\alpha + \beta})} n^{\frac{\alpha}{2\alpha + \beta}} + \frac{\alpha}{\alpha +\beta} + O \left(n^{\frac{\alpha}{2\alpha + \beta} -1}\right) \, ,
\end{equation}
\begin{equation}
{\mathbb{V}}(X_n) = \frac{\alpha^3 (2\alpha + \beta)}{(\alpha + \beta)^2} n + O\left(n^{\frac{\alpha+\beta}{2\alpha+\beta}}\right) \ \mbox{ where } \rmgamma (x) := \int_0^{\infty} t^{x-1} e^{-t} \text{\emph{d}}t \, .
\end{equation}
\end{proposition}

\begin{proof}
The following techniques are described in \cite{FlaSed09}.
We differentiate the equation~(\ref{eqn_algebric}) with regards to $x$, then we set $x=1$;  we use the asymptotic expansion form and normalize by the asymptotic expansion of $(1-\sigma z)^{-1/\sigma}=: H(1,1,z)$.
For variance, we differentiate twice and we use the same technique. It can be done for all moments. Besides, it is possible to express any term in the asymptotic expansion, thanks to Dr. Salvy's {\sc Maple} package \texttt{gfun}.\footnote{available on \texttt{http://algo.inria.fr/libraries/}.} \qed
\end{proof}

%
%
\section{Asymptotic Results}\label{limit_th}

The aim of this section is to obtain properties on the limit distribution of the number of black balls in the urn. Our combinatorial point of view implies the study of $z$-coefficients of the generating function $H(x,z)$.
%
\subsection{Probability Expression and the Case $x=1$}
\begin{lemma}\label{yn_exp}
The $z^n$-coefficient in the generating function of histories $H(x,z)$ can by expressed by a contour integral (with a contour around the origin):
\begin{equation}\label{lem_yn}
[z^n] H(x,z) = \frac{\sigma^{n+1}}{2 i \pi} \oint a_x(w) h_x(w)^{n+1} \text{\emph{d}}w \, ,  \qquad\text{where}
\end{equation}
\begin{align}
 h_x(w) & = \left[ 1 + \frac{\sigma \left( x^{-\alpha} -1 \right) }{\alpha+\beta}     - (1-w)^{\alpha+\beta} \left(  \frac{\sigma\left( x^{-\alpha} -1 \right)}{\alpha+\beta}    + (1-w)^\alpha \right) \right]^{-1}  ,\\
  a_x(w) & = (1-w)^{\alpha + \beta -2} \left(   x^{-\alpha} -1 + \left( 1-w \right)^{\alpha}   \right) \, .
\end{align}
\end{lemma}

\begin{proof}
We start with the polynomial (\ref{eqn_algebric}) which cancels $y(x,z):=H(x,z)$. We use a Cauchy formula to express the $z^n$-coefficient,
\begin{equation}\label{cauchy}
[z^n] H(x,z) = y_n(x)  =\frac{1}{2 i \pi} \oint \frac{y(x,z)}{z^{n+1}}\text{d}z \, .
\end{equation}
Then, by a Lagrange inversion\footnote{For details on Lagrange inversion, see \cite{FlaSed09} (Appendix A.6, {\it p. 732}).} between $z$ and $y(z)$, using \mbox{the equation~(\ref{eqn_algebric}),}
\begin{equation}
 z  = \sigma^{-1} + B_x - y^{-\sigma} \left( \sigma^{-1} + B_x y^{\alpha} \right) \, .
\end{equation}
At this step, it is not yet possible to apply a saddle-point method, because the saddle-points are at $\infty$. Again, we change variables, and set $w = 1 - y^{-1}$ to get the result. \qed
\end{proof}

From this lemma, we have to evaluate this integral in order to have an expression of $[z^n] H(x,z)$. It will be necessary for describing the limit probability distribution. Indeed, our main goal is to understand the behavior of the probability generating function noted $p_n(x)$,
\begin{align}
 p_n(x) = \sum_{i\geq 0} {\mathbb{P}}\left( X_n = i\right) x^i = \frac{[z^n] H(x,z)}{[z^n] H(1,z)} \, .
 \end{align}
First, we focus on the evaluation of \mbox{$ [z^n] H(1,z) $}. 
In the special case $x=1$,  we can solve explicitly the equation (\ref{eqn_algebric}). We have \mbox{$H(1,z) = (1 - \sigma z)^{-1/\sigma}$}. Thus, we know with classical analytic combinatorics that \mbox{$ y_n(1) \sim  n^{-\frac{\sigma - 1}{\sigma} }{\rmgamma \left( \sigma^{-1} \right)}^{-1}  \sigma^n $}. 
For $x \neq 1$, we don't have access to the explicit solution, therefore we use a saddle-point method for the general case. In the Appendix \ref{saddle-point}, the simple case $x = 1$ is treated with this method. It is useful to guess the right normalization for the variables and the contour for the general case. For illustration, Fig.~\ref{sadd3} shows the behavior of $h_1(w)$ and how to choose the contour.

\begin{figure}[htbp]
\begin{center}
\includegraphics[height=6cm]{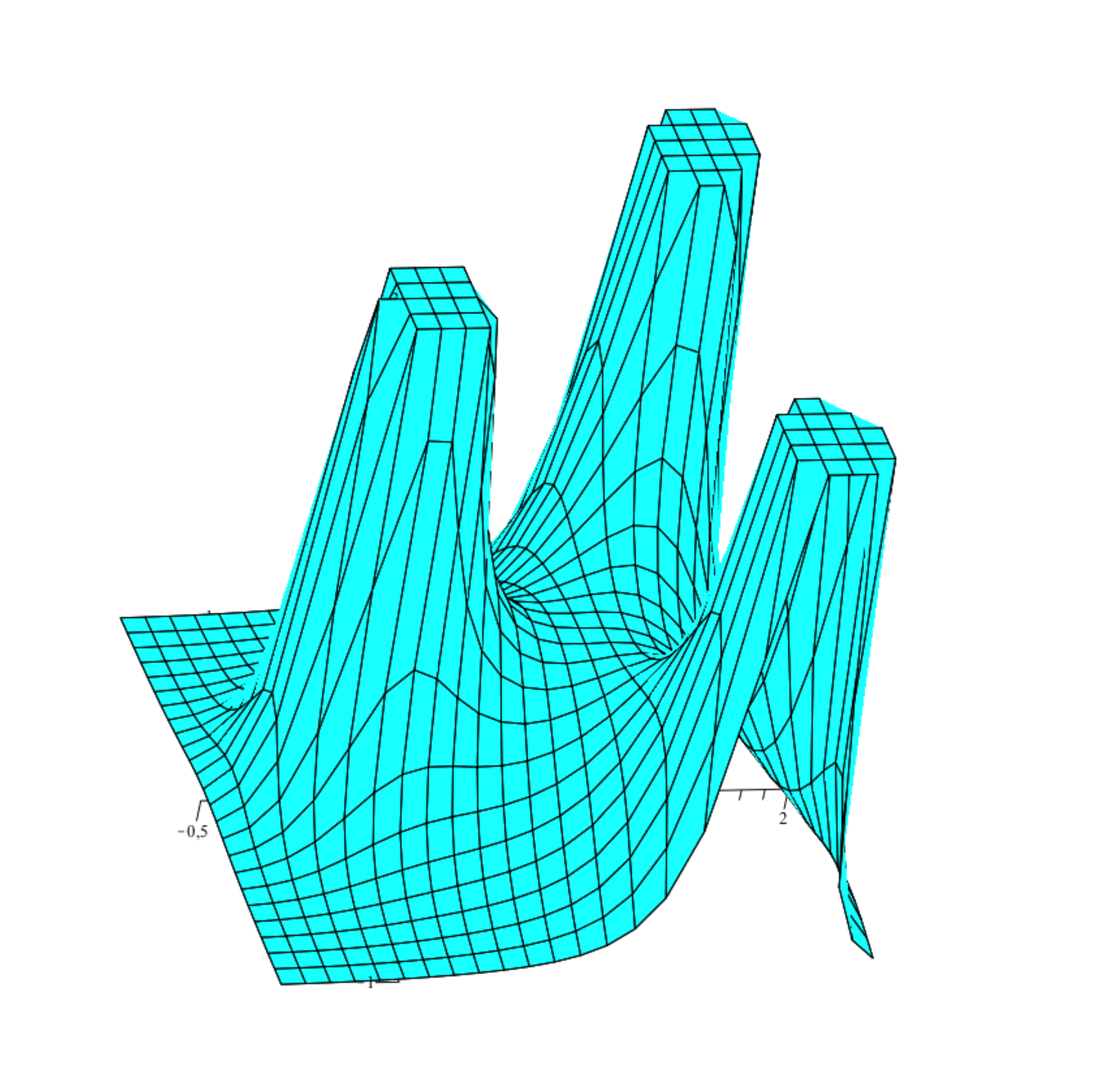}
\includegraphics[height=4cm]{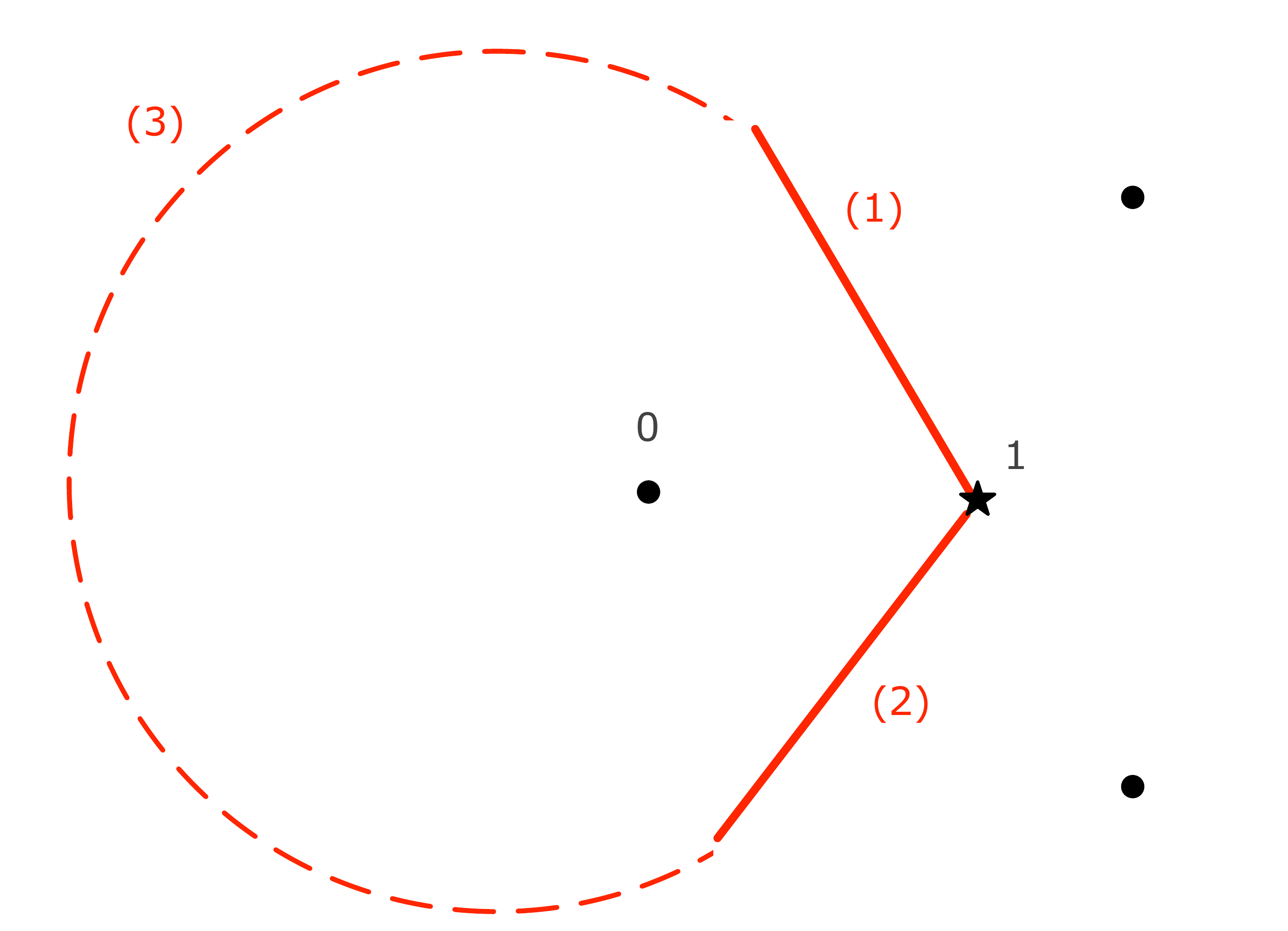}
\caption{{\it Left -} Double saddle-point for $h_1(w)$ for the urn $\mathcal{A}(1,1)$. The three {\it peaks} correspond to zeros of the denominator of $h_1(w)$ which are $0$, $1-j$ and $1-j^2$, where $j$ is the third-root of unity $\exp(2i\pi/3)$. {\it Right -} Diagram from a top view of the three  poles and the double saddle-point in 1, with the three-parts contour $\mathcal{C}_1$, $\mathcal{C}_2$, $\mathcal{C}_3$.}
\label{sadd3}
\end{center}
\end{figure}

\subsection{Limit Theorems}\label{gen_case}
Here is the main result of the paper. It is described in the three following theorems which concerns the asymptotic distribution of the balls in the urn. Proposition~\ref{y_expr} expresses the asymptotic expansion of the probability generating function $p_n(x)$. The refined saddle-point analysis is the core of the proof of this proposition. Then, the proofs of the theorems are based on this proposition and on Quasi-Power theorems from H.-K.Hwang, recalled in \cite{FlaSed09} {\it (p. 645, 696, 700)}.

Let $X_n$ be the random variable counting the number of black balls in the urn $\mathcal{A}(\alpha, \beta)$ after $n$ steps.
\begin{theorem}\label{th_loi_limite}{(Gaussian limit law)}
The random variable $X_n$ has mean $\mu_n$ and variance $\nu_n^2$, and the normalized random variable  $\frac{X_n - \mu n}{\nu \sqrt{n}}$ converges in law to the standard normal law  ${\mathcal{N}}(0,1)$, with rate of convergence $O\left( \frac{1}{\sqrt{n}} \right)$,
$$ \mu_n = {\mathbb{E}} (X_n) = \mu n + o(n) \ , \ \ \nu_n^2  = {\mathbb{V}} (X_n) = \nu^2 n  + o(n) \, ,$$
$${\mathbb{P}} \left\{  \frac{X_n - \mu n}{\nu \sqrt{n} }  \leqslant t \right\} = \Phi (t) + O \left( \frac{1}{\sqrt{n}}\right) \, , \ \text{where } \ \Phi (t) = \frac{1}{\sqrt{2 \pi}} \int_{-\infty}^t e^{-\frac{v^2}{2}} \text{\emph{d}}v \, ,$$
with  $$ \mu = \frac{\alpha (2\alpha+\beta)}{\alpha+\beta}  \qquad \text{and}\qquad  \nu^2  = \frac{\alpha^3 (2\alpha+\beta)}{\left( \alpha + \beta \right)^2} \, .$$
\end{theorem}
\begin{theorem}\label{th_loi_locale}{(Local limit law)}
 We denote $p_{n,k} = {\mathbb{P}}\left\{ X_n =k \right\}$. The distribution of $X_n$ satisfies a local limit law of Gaussian type with rate of convergence $O\left( \frac{1}{\sqrt{n}} \right)$,
i.e. $$ \sup_{t \in {\mathbb{R}}} \left|  \nu \sqrt{n} p_{n, \lfloor \mu n  + t \nu \sqrt{n} \rfloor}  - \frac{1}{\sqrt{2 \pi }}  e^{-t^2/2} \right|  \leqslant \frac{1}{\sqrt{n}}  \, .$$
\end{theorem}

\noindent The last theorem concerns large deviation bounds.
\begin{definition} (Large deviations property)
Let $\left(a_n\right)$ be a sequence tending to infinity. 
A sequence of random variables $(X_n)$ with mean ${\mathbb{E}}X_n \sim \mu a_n$, satisfies a \emph{large deviation property}, relative to the interval $[t_0,t_1]$ containing $\mu$, if there exists a function $W(t)$ such as $W(t)>0$ for $t \neq \mu$, and for $n$ large enough,
$$ \forall t, \ t_0 < t < \mu ,  \ \frac{1}{a_n} \log {\mathbb{P}}(X_n \leqslant t a_n) = - W(t) + o(1)  \ \ \  \text{\textsl{(left tail)}} \, ,$$
$$ \forall t, \  \mu < t < t_1 ,  \ \frac{1}{a_n} \log {\mathbb{P}}(X_n \geqslant t a_n) = - W(t) + o(1)  \ \ \  \text{\textsl{(right tail)}} \, ,$$
W(t) is called the \emph{rate function}, and $a_n$ is the \emph{scale factor}.
\end{definition}
\begin{theorem}\label{grandes_deviations}{(Large deviations)}
For $\xi$ such as $0 < \xi < 1$, the sequence of random variables $(X_n)$ satisfies a large deviation property relative to the interval $[t_0,t_1]$, with a scale factor $n$, and a rate $W(t)$, with
$x_0 = \xi, \ x_1 = 2 -\xi \, ,$
$$t_0 = x_0 \frac{\chi'(x_0)}{\chi(x_0)}  \ \text{ and } \ t_1 =  x_1 \frac{\chi'(x_1)}{\chi(x_1)} \, ,$$
$$ W(t) = - \min_{x \in  \left[ x_0, x_1 \right] } \ \log \left( \frac{\chi(x)}{x^t} \right) \, , \text{ where } \chi(x) = x^{\mu} \, e^{ \frac{\nu^2}{2}  \ln(x)^2} \, .$$ 
\end{theorem}

\noindent The main proposition in the proof of the theorems is the asymptotic expression of the probability generating function $p_n(x)$, obtained by the division of $y_n(x)$ by $y_n(1)$, where $y_n(x)$ denotes the $z^n$-coefficient in $H(x,z)$.
\begin{proposition}\label{y_expr}

For $x =  e^{i u /\sqrt{n}}$
,  the expression of the probability generating function $p_n(x)$ is, asymptotically when $n \rightarrow \infty$,
\begin{align}\label{pn_expr}
 p_n(x) =  \exp \left( \mu  i u \sqrt{n} - \frac{\nu^2}{2} u^2 \right) \left( 1 + O\left( u n^{-\frac{\beta}{2(2\alpha+\beta)}}\right)\right) \, ,
 \end{align}
with $$ \mu = \frac{\alpha (2\alpha+\beta)}{\alpha+\beta}  \qquad \text{and}\qquad  \nu^2  = \frac{\alpha^3 (2\alpha+\beta)}{\left( \alpha + \beta \right)^2} \, .$$ 
\end{proposition}

\begin{remark} This expression (\ref{pn_expr}) of $p_n(x)$ is valid
for $x$ in a neighborhood of $1$,  with the radius $n^{-1/2}$ (which will give the Gaussian limit law);
for $x$ in  a real segment centered in $1$ of length $n^{-1/2}$  (which will give large deviation);
for $x$ in the unit circle,  $|x| = 1$  (which will give a local limit law).
\end{remark}

\subsection{Proof of Proposition~\ref{y_expr}}\label{mainproof}
\begin{proposition}
For all $x$, $h_x(w)$  has a simple pole in $w=0$, and $2 \alpha + \beta - 1$ other simple poles.
Saddle-points are in 
$w=1$ with multiplicity $\alpha + \beta -1$,
and in $w =  1- \gamma$, with $\gamma$ any  $\alpha$th-root of $(1-x^{-\alpha})$.
\end{proposition}

\begin{proof} 
\noindent In order to find saddle-points, we look at zeros of the derivative of $h_x$,
$$
h'_x(w)   =  - \sigma h_x(w)^2  \left( 1-w \right)^{\alpha+\beta -1} \left(  x^{-\alpha} -1 + \left( 1-w \right)^{\alpha} \right) \, .
$$
\end{proof}

\begin{remark}
In our case, when  $x \neq 1$, we have $\alpha +1$ different saddle-points: one ``huge'' at $w=1$, and $\alpha$ secondary saddle-points depending on the $x$ value. When $x$ is in a neighborhood of $1$, the $\alpha$ saddle-points will merge at the point $w=1$. We have to deal with $\alpha +1$ coalescing saddle-points. It differs from usual saddle-points method because it is no longer possible to treat each saddle-point independently.
We need to consider all saddle-points together in order to find the asymptotic expansion.\end{remark}

\begin{figure}[htbp]
\begin{center}
\includegraphics[height=5.5cm]{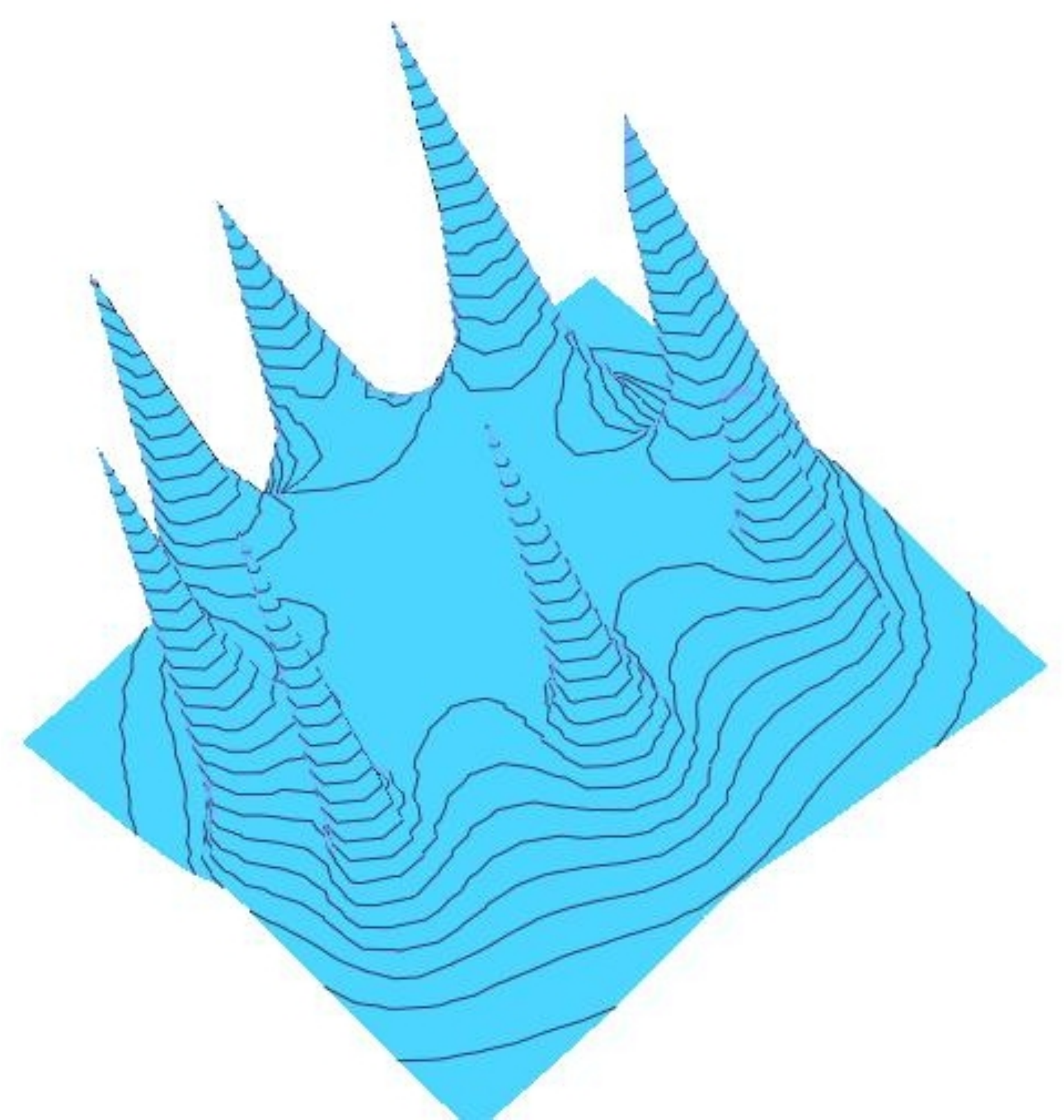}
\includegraphics[height=4cm]{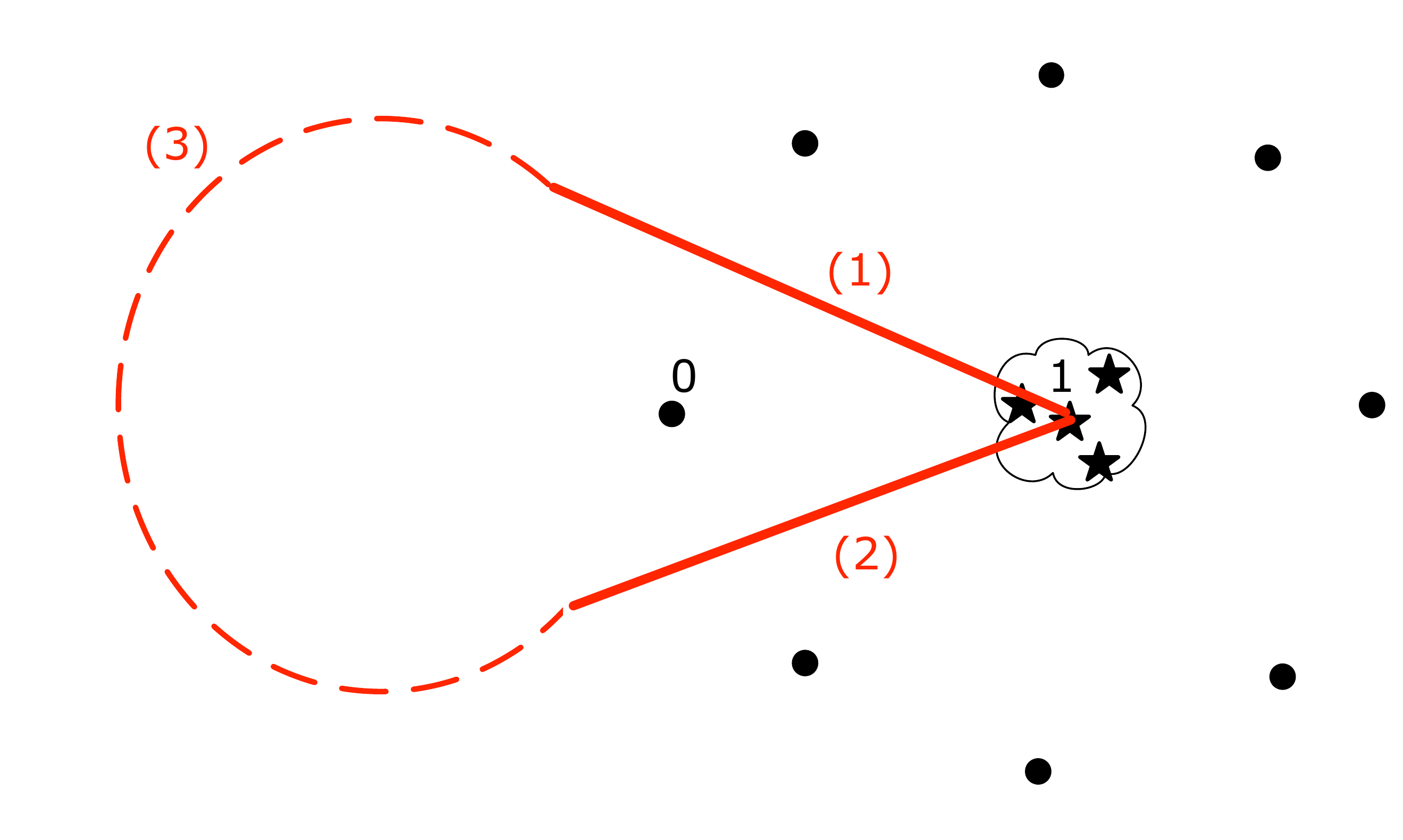}
\caption{{\it Left -} Graph of $|h_x(w)|$ for the urn $\mathcal{A}(3,2)$. There are $\sigma=8$ poles, and  4 saddle-points at point $1$ and at $1 - \left(1-x^{-3}\right)^{1/3}$.
{\it Right -} Multiple saddle-points for $h_x(w)$. The {\it peaks} (poles) corresponds to  the zeros of the denominator of $h_x(w)$. The {\it stars} represent the different saddle-points. We illustrate {\it in red} the integration contour around 0.}
\label{sadd_multiple}
\end{center}
\end{figure}

We expect a Gaussian law (thanks to probabilistic results), so we use a renormalisation for $x$ of order $n^{-1/2}$. 
For the renormalisation of $w$, we use the same scale as previously in the particular case $x=1$ (see Appendix \ref{saddle-point}). So the normalizing factor is~$n^{-1/\sigma}$.

The general idea for the choice of the contour is the same as in the previous simple case $x=1$. We will use two segments beginning in the saddle-point area in $w=1$ and they have to follow the descents of the surface. This descent has to go between the pole $0$ and the first closest pole. The segment will make an angle of $\frac{\pi (\sigma-1)}{\sigma} = \pi - \frac{1}{2}\frac{2 \pi}{\sigma} $, as we can see on the right part of Fig.~\ref{sadd_multiple}.

\begin{definition}The integration contour is described in three parts.
\noindent We use the following parametrization for the segment $\mathcal{C}_1$. For $t \in [0,\, n^2]$,
\begin{align} 
w & = 1 + \left(\frac{t}{n}\right)^{1/\sigma} e^{i \pi \frac{\sigma -1}{\sigma}} \, ,\\
x & = e^{i u /\sqrt{n}} \, .
\end{align}
The same parametrization is used for the segment $\mathcal{C}_2$, with an angle of $-\pi\, \frac{\sigma -1}{\sigma}$.
The circle part $\mathcal{C}_3$ completes the contour.
\end{definition}
\begin{proposition}\label{hx_exp}
There is an asymptotic expansion in $n$ for $h_x(w)^n$:
$$ h_x(w)^n = e^{ \mu  i u \sqrt{n} - \frac{\nu^2}{2} u^2 }  e^{-t} \exp \left(  K t^{\frac{\alpha+\beta}{(2\alpha+\beta)}} u n^{\frac{- \beta}{ 2 (2\alpha+\beta)}}  + O\left(  n^{-1/2}  \left( i u^3 + u t\right) \right)  \right)  ,$$
with $$ \mu = \frac{\alpha (2\alpha+\beta)}{\alpha+\beta} \, , \qquad  \nu^2   = \frac{\alpha^3(2\alpha+\beta)}{\left( \alpha + \beta \right)^2} \qquad \text{and}\qquad x=e^{i u / \sqrt{n}} \, .$$ 
\end{proposition}
\begin{remark} Notice that this expansion is valid for all $t$ and $u$. Besides, we obtain the same expansion using for $x$ the expression
$ x =  1 + i u n^{-1/2}  - u^2 n^{-1} $.
Then, it will be possible to use this expansion for all $|u| < 1$, then we will have an expression for $|x| < \frac{1}{\sqrt{n}}$ uniformly, which is a necessary hypothesis for Hwang's theorem on limit law.  We will use $x$ real near from $1$ for large  deviations theorem.
The expression with $x=e^{i u/\sqrt{n}}$ implies $|x| =1$ for the local limit law theorem.
\end{remark}
\begin{lemma}\label{tail_cut} [Neglect the tails]
The $\mathcal{C}_1$-contribution is asymptotically equivalent on both intervals, for $t \in [0, \, n^2]$ and for  $t \in \left[0, \, n^{\frac{1}{\sigma +1}}\right]$. The error term is $O \left( \exp \left( - n^{\frac{1}{\sigma +1}}\right)\right)$, so it is exponentially negligible.
\end{lemma}
\begin{proof}
On the path $\mathcal{C}_1$, the function $h_x(t)$ is strictly decreasing for \mbox{$t>  n^{\frac{1}{\sigma+1}}$}. \qed
\end{proof}

\begin{lemma}\label{C1_part} [Central approximation]
The $\mathcal{C}_1$-contribution can be written
$$
\frac{\sigma^n}{2 i \pi} e^{i \pi /\sigma} n^{- \frac{\sigma -1}{\sigma}}  \exp \left( \mu  i u \sqrt{n} - \frac{\nu^2}{2} u^2 \right) \int_0^{\frac{1}{\sigma +1}}  t^{-1/\sigma} e^{-t} \left( 1 + t^{\frac{\alpha +\beta}{\sigma}}O \left( u n^{-\frac{\beta}{2\sigma}} \right) \right)  \text{\emph{d}}t .
$$
\end{lemma}
\begin{proof}
We use Lemma \ref{tail_cut} to restrict the interval. Then we can apply the asymptotic expansion of Proposition \ref{hx_exp}. Besides, we have to know an asymptotic expansion for $a_x(w)$ on this path, and we have
$$ a_x(t) \sim   \left(\frac{t}{n}\right)^{(\sigma -2)/\sigma}  \, e^{-i \pi(\sigma -2)/\sigma } \, .$$
\end{proof}

\begin{lemma}\label{tail_complete} [Complete the tails]
The integral of $\, t^{-1/\sigma}e^{-t} \, $ on segment $[0, \, n^{\frac{1}{\sigma+1}}]$ is equivalent to the same integral on interval $[0, \, \infty[$. \\ The error term is $O \left( \exp \left( - n^{\frac{1}{\sigma +1}}\right)\right)$, so it is exponentially negligible.
\end{lemma}

\begin{lemma}\label{2-parts}
The $\mathcal{C}_1$ and $\mathcal{C}_2$ contribution can be written
\begin{align}
\frac{\sigma^n}{\rmgamma\left( \frac{1}{\sigma}\right)}n^{\frac{1}{\sigma}-1}   \exp \left( \mu  i u \sqrt{n} - \frac{\nu^2}{2} u^2 \right) \left( 1 + O\left( u n^{-\frac{\beta}{2\sigma}}\right)\right) \, .
\end{align}
\end{lemma}
\begin{proof}
The work on segment $\mathcal{C}_1$ is the same for segment $\mathcal{C}_2$. Adding the two contribution, we get a $\left(e^{i \pi / \sigma} - e^{-i \pi / \sigma}\right)$ factor. Thanks to Lemma \ref{tail_complete}, the integral part of calculus reduces to the simple Gamma factor $\rmgamma \left( \frac{\sigma -1}{\sigma}\right)$. \qed
\end{proof}

\begin{lemma}\label{3-part}
The circle part $\mathcal{C}_3$ is exponentially negligible.
\end{lemma}
Adding all the contributions, thanks to Lemmas \ref{2-parts} and \ref{3-part}, we obtain a global expression of $y_n(x)$,
\begin{align} y_n(x) =  \frac{\sigma^n}{\rmgamma\left( \frac{1}{\sigma}\right)}n^{\frac{1}{\sigma}-1}   \exp \left( \mu  i u \sqrt{n} - \frac{\nu^2}{2} u^2 \right) \left( 1 + O\left( u n^{-\frac{\beta}{2\sigma}}\right)\right) \, .
\end{align}

\noindent We can express $p_n(x)$ with $x$ in a neighborhood of $1$ with radius $n^{-1/2}$, and we obtain an expression suited to Quasi-Power Theorems,
\begin{align} p_n(x) = \left( x^{\mu} \exp \left( \frac{\nu^2}{2} \ln(x)^2 \right) \right)^n \left( 1 + O \left(    n^{\frac{- \beta}{2 \sigma}} \right)\right) \, .\end{align}
For the three main theorems, we use this expression which satisfies the hypothesis for applying H.-K.Hwang theorems: Quasi-Power Theorem, Quasi-Power Limit Theorem and Quasi-Power for large deviations (in \cite{FlaSed09}, {\it p.645, 696,~700}).


\section{Conclusion}

Through the use of analytic combinatorics, we obtain precise probability results, including a Gaussian limit law with its rate of convergence, a local limit law and large deviation bounds. For the first time, this kind of probabilistic results on additive balanced urns are obtained from an analytic point of view. As previously thought, analytic combinatorial methods can provide a wealth of valuable information. In addition, our results seem to reinforce the idea of a general theory for additive balanced urns. In this way, the work on urns linked to $k$-trees in \cite{Morcrette10} use an other subclass of additive balanced urns, and similar probabilistic results were obtained. The natural next step is having a complete understanding of the whole class, following ideas of Flajolet et al. in \cite{FlaDuPuy06}.

\subsubsection*{Acknowlegment.}
I warmly thank my advisor and mentor, Philippe Flajolet, for guiding me and introducing me to research and especially to analytic combinatorics. This paper is dedicated to him.
In addition, I'm grateful to J\'er\'emie Lumbroso for his helpful remarks.\\
This work was supported by the ANR project 09 BLAN 0011 Boole  and the ANR~project 10 BLAN 0204 Magnum.

\bibliographystyle{splncs03}

\appendix
\section{Appendix - Saddle-point method when $x=1$}\label{saddle-point}

We detail here the evaluation of the contour integral $y_n(1)$ for the simple case $x=1$,
\begin{align}
y_n(1) 
= \frac{\sigma^{n+1}}{2 i \pi} \oint \frac{\left( 1-w\right)^{\sigma -2}\text{d}w}{\left( 1 - (1-w)^\sigma\right)^{n+1}} \, .
\end{align}
Then,  $h_1(w) = \left( 1 - (1-w)^\sigma \right)^{-1}$ and $a_1(w) = (1-w)^{\sigma -2}$.
We have now to evaluate the integral. 
This is classical saddle-point method. The derivative $h_1'$ becomes zero only for one value $w=1$. There is only one saddle-point. We will use two segments starting at $w=1$ and descending along the surface, passing between the peaks (singular points of $h_1$). The third part of the contour is a circle to join the two ends of segments.
Figure~\ref{sadd3} illustrates the three parts of the contour for the urn $\mathcal{A}(1,1)$.

\noindent First we look at the parametrization of the two segments:
\begin{itemize}
\item  $\mathcal{C}_1$ :  $ w = 1 + t e^{i \pi\frac{\sigma-1}{\sigma}} $, with $t \in [0, \, L]$,
\item  $\mathcal{C}_2$  : $ w = 1 + t e^{- i \pi \frac{\sigma-1}{\sigma}} $, with $t \in [0, \, L]$.
\end{itemize}
The two cases are similar. For $\mathcal{C}_1$,  the integral rewrites
\begin{align}
\frac{\sigma^{n+1}}{2 i \pi} \left( - e^{-i \pi/\sigma}\right)  \int_0^L - t^{\sigma-2} e^{2i \pi /\sigma} \left( \frac{1}{1+t^\sigma}\right)^{n+1} \text{d}t.
\end{align}
Adding the two contributions of the two paths $\mathcal{C}_1$ and $\mathcal{C}_2$, we get
\begin{align}\frac{\sigma^{n+1}}{\pi} \sin\left( \frac{\pi}{\sigma}\right) \int_0^L \frac{t^{\sigma - 2} \text{d}t}{\left( 1 + t^\sigma \right)^{n+1}}.\end{align}
The expansion $ h_1(t)^n  =  \exp \left( -n t^\sigma + O \left( nt^{2 \sigma}\right)  \right)$ leads us to change variable with $u = n t^\sigma$. Thus, the integral rewrites
 \begin{align}  \int_0^L \frac{t^{\sigma - 2} \text{d}t}{\left( 1 + t^\sigma \right)^{n+1}}   =  \frac{n^{- \frac{\sigma - 1}{\sigma}}}{\sigma} \int_0^{nL^\sigma} u^{-1/\sigma}  \exp  \left( -u  + O \left( u^2 n^{-1} \right)  \right) \text{d}u \, .
 \end{align}
We have to adjust the parameter $L$ which is the length of each segment. In order to use the expansion approximation, some conditions impose themselves~:  first, $nL^\sigma \rightarrow \infty$, and second,  $nL^{2 \sigma} \rightarrow 0$.
We decide to fix $L \sim n^{\frac{-1}{\sigma +1}}$. Thus, $n L^{\sigma} = n^{\frac{1}{\sigma +1}}$  and $nL^{2\sigma}  = n^{- \frac{\sigma-1}{\sigma+1}}$.
In this case, we complete the tail of the Gaussian approximation,
\begin{eqnarray}
\int_0^{n^{1/(\sigma+1)}} \hspace{-0.5cm} u^{-1/\sigma}  e^{ -u  + O \left( u^2 n^{-1}  \right)} \text{d}u \ & \ \sim \ & \ \int_0^{n^{1/(\sigma+1)}} \hspace{-0.5cm} u^{-1/\sigma}  e^{-u} \text{d}u + O\left(n^{-1}\right) \ \\
&\  \sim \ &\  \int_0^{\infty} u^{-1/\sigma}  e^{-u} \text{d}u + O\left(n^{-1}\right)\  \\
& \ \sim\  & \ \rmgamma \left(\frac{\sigma -1}{\sigma}\right) + O\left(n^{-1}\right) \, .
\end{eqnarray}
To conclude with this illustration example, we need to deal with the third part of the contour, which is the circle part.
We show that this part is exponentially small. Indeed, it is possible to stretch the segments ${\mathcal{C}_1}$
and ${\mathcal{C}_2}$. Thus, the radius of the circle part will grow to infinity and this integral part will be negligible.
Besides, it is effectively possible to stretch the two segments because the function $h_1$ is strictly decreasing along the segments. Its maximum is at the saddle-point.
On $[0, \, n^\sigma]$, we have $h_1'(t) = -\sigma t^{\sigma -1} (1 + t^\sigma)^2  < 0$. 
So, the evaluation of the integral on the segment $[0, \, n^\sigma]$ is the same order as the evaluation on  $[0, \, n^{1/(\sigma+1)}]$. Indeed, the majoration error is exponentially small. This error is about the order $O \left( \exp (- n^{1/(\sigma +1)}) \right)$, which is negligible.

\noindent Gathering all parts, we obtain
\begin{eqnarray}
 y_n &=& \frac{\sigma^{n+1}}{\pi} \sin\left(\frac{\pi}{\sigma}\right)  \frac{n^{-\frac{\sigma -1}{\sigma}}}{\sigma} \rmgamma \left( \frac{\sigma -1}{\sigma}\right)  \left( 1 + O\left( n^{-1}\right)\right)\\
 &=& \frac{\sigma^n \, n^{1-1/\sigma}}{\rmgamma(1/\sigma)} \left( 1 + O\left( n^{-1}\right)\right)\, .
\end{eqnarray}

\end{document}